\documentclass[11pt,reqno,a4paper]{amsart}
\usepackage[utf8]{inputenc}
\usepackage[english]{babel}
\usepackage[left=1in, bottom=1in, top=1in, right=1in,footskip=0.5in]{geometry}
\usepackage{amsmath}
\usepackage{amssymb}
\usepackage{amsthm}
\usepackage[dvipsnames]{xcolor}
\usepackage[pdftex, colorlinks, linkcolor={PineGreen!80!black}, citecolor={PineGreen!80!black}, urlcolor={PineGreen!80!black}, pdfpagemode=UseOutlines, bookmarksopen=true]{hyperref}
\usepackage[alphabetic, msc-links, nobysame]{amsrefs}
\usepackage[capitalize,nameinlink]{cleveref}

\theoremstyle{plain}
\newtheorem{thm}{Theorem}[section]
\newtheorem{lem}[thm]{Lemma}
\newtheorem{prop}[thm]{Proposition}
\newtheorem{cor}[thm]{Corollary}
\newtheorem*{conj}{Conjecture}
\theoremstyle{definition}

\crefname{thm}{Theorem}{Theorems}
\crefname{lem}{Lemma}{Lemmas}
\crefname{conj}{Conjecture}{Conjectures}

\numberwithin{equation}{section}

\newcommand{\R}{\mathbb{R}}
\newcommand{\mmod}[1]{~\!\!(\mathrm{mod}~#1)}

\let\Re\relax
\let\Im\relax
\DeclareMathOperator{\Im}{Im}
\DeclareMathOperator{\Re}{Re}

\renewcommand{\L}{\mathcal{L}}
\newcommand{\glh}{\hyperlink{conj-glh}{GLH}}

\begin{document}

\title{A conditional bound for the least prime in an arithmetic progression}
\author{Mat\'ias Bruna}
\address{Department of Mathematics, University of Toronto, Toronto, ON, M5S 2E4, Canada}
\email{matias.bruna@mail.utoronto.ca}

\begin{abstract}
Assuming the generalized Lindel\"of hypothesis for Dirichlet $L$-functions, we establish that the least prime  $p\equiv a\pmod{q}$ satisfies $p\ll_{\varepsilon} q^{2+\varepsilon}$.
This achieves a bound that nearly matches the classical estimate implied by the generalized Riemann hypothesis.
\end{abstract}

\maketitle

\section{Introduction}

Given $a$ and $q$ coprime integers, let $P(a,q)$ be the least prime $p\equiv a\pmod{q}$.
In 1944, Linnik \cites{Linnik-1,Linnik-2} proved the existence of an absolute constant $L>0$ such that
\begin{equation}\label{eq-Linnik-L}
P(a,q)\ll q^{L},
\end{equation}
establishing the first unconditional polynomial upper bound in $q$ valid for all residue classes.
Assuming the generalized Riemann hypothesis for Dirichlet $L$-functions (GRH), one can prove
\begin{equation}\label{eq-Linnik-grh}
P(a,q)\ll(\phi(q)\log q)^{2},
\end{equation}
which implies \eqref{eq-Linnik-L} with $L=2+\varepsilon$.
While this bound is classical, a sequence of works has made the implied constant explicit, starting with Bach \cite{Bach-1985} and recently refined by Lamzouri, Li, and Soundararajan \cite{Lamzouri-Li-Sound-2015} and Carneiro, Milinovich, Quesada-Herrera, and Ramos \cite{CMQHR}.

The goal of this article is to establish a bound for $P(a,q)$ under the weaker assumption of the generalized Lindel\"of hypothesis for Dirichlet $L$-functions (GLH).
In its classical form, GLH states that for any primitive Dirichlet character $\chi$ modulo $q$, for all $\varepsilon>0$ and $t\in\R$, we have
\[L(1/2+it,\chi)\ll_{\varepsilon}(q(1+|t|))^{\varepsilon}.\]
By considering the primitive character $\chi^{*}$ inducing $\chi$, this bound extends to all characters.
The Phragmen--Lindel\"of principle then extends this to the half-plane $\sigma\geq 1/2$, provided $s=\sigma+it$ stays bounded away from the pole at $s=1$ when $\chi$ is principal.
We therefore adopt the following form of the conjecture.
\hypertarget{conj-glh}{}
\begin{conj}[GLH]
Let $\chi\pmod{q}$ be a Dirichlet character.
Then for any $\varepsilon>0$ and $s=\sigma+it$ satisfying $\sigma\geq 1/2$ and $|s-1|\geq 1/2$, we have
\[L(\sigma+it,\chi)\ll_{\varepsilon}(q(1+|t|))^{\varepsilon}.\]
\end{conj}

Our main theorem shows that a bound nearly matching \eqref{eq-Linnik-grh} holds under GLH.

\begin{thm}\label{main-thm}
Assume \glh.
For any $\varepsilon>0$ and any integer $a$ coprime to $q$, there exists a constant $q(\varepsilon)$ such that for all $q\geq q(\varepsilon)$ we have $P(a,q)\leq q^{2+\varepsilon}$.
The constant $q(\varepsilon)$ is effectively computable for $\varepsilon>1$ and ineffective otherwise.
\end{thm}

The proof follows a classical strategy using three key principles of Dirichlet $L$-functions: A zero-free region, log-free zero-density estimates, and a zero-repulsion theorem.
Under \glh, we obtain an arbitrarily large zero-free region and the density hypothesis in the entire half-strip.
By proving the latter, we establish a result that, to the author's knowledge, is considered folklore but has previously lacked a complete proof in the literature.
Together, these results allow us to bypass the delicate numerical calculations typically required in the unconditional setting.

Regarding the ineffectivity of the constant $q(\varepsilon)$, this arises specifically when a Siegel zero is exceptionally close to $s=1$.
In this extreme regime, we rely on the work of Heath-Brown \cite{Heath-Brown-Siegel-Linnik} concerning exceptional zeros and the least prime in an arithmetic progression.
Concretely, the ineffective exponent $2+\varepsilon$ stems from \cite{Heath-Brown-Siegel-Linnik}*{Corollary 2}, while \cite{Heath-Brown-Siegel-Linnik}*{Corollary 1} gives an effective upper bound with exponent $3+\varepsilon$.
To secure an effectively computable constant in the remaining subregime of this exceptional setting, we leverage the recent explicit Deuring--Heilbronn phenomenon established by Benli, Goel, Twiss, and Zaman \cite{Benli-Goel-Twiss-Zaman}.

Unconditionally, the strongest result obtained via the classical approach described above is due to Xylouris \cite{Xylouris-5}, who, building on ideas of Heath-Brown \cite{Heath-Brown-Linnik}, proved in his PhD thesis that one may take $L=5$.
See \cite{Heath-Brown-Linnik}*{\S 1} and the references therein for further background and historical context on these developments.

Alternative approaches avoiding the classical $L$-function machinery have also been developed. Most notably, Friedlander and Iwaniec \cites{Friedlander-Iwaniec-Linnik-1,Friedlander-Iwaniec-Linnik-2} used classical sieve methods to prove Linnik’s theorem with $L=75~\!744~\!000$, whereas Matomäki, Merikoski, and Teräväinen \cite{Matomaki-Merikoski-Teravainen-mult-structured} developed a new prime-detecting sieve that gave the much stronger bound $P(a,q)\ll q^{350}$.
In a different direction, Granville, Harper, and Soundararajan \cite{Granville-Harper-Sound-new-Halasz} proved Linnik's theorem using the so-called ``pretentious'' number theory, though without an explicit exponent $L$.

\subsection*{Notation}

Let $q\geq 3$ be an integer and set $\L=\log q$.
Given functions $f$ and $g$, we write $f\ll g$ or $f=O(g)$ to mean that there exists a positive constant $C$ such that $|f(x)|\leq Cg(x)$ for all $x$ in the range under consideration.
All implied constants are absolute unless explicitly stated, and dependence on a parameter such as $p$ is indicated by $\ll_{p}$ or $O_{p}$.
We write $s=\sigma+it$, and given $\chi\pmod{q}$ a Dirichlet character, we denote by $L(s,\chi)$ the corresponding $L$-function.
Its non-trivial zeros are denoted by $\rho=\beta+i\gamma$, where $\beta=1-\lambda/\L$.

\section{The three principles}

\subsection{Zero-free region}

We begin by recalling a Jensen-type formula due to Heath-Brown.

\begin{lem}[\cite{Heath-Brown-Linnik}*{Lemma 3.2}]\label{HB-jensen}
Let $f(z)$ be holomorphic for $|z-s|\leq R$, and non-vanishing both at $z=s$ and on the circle $|z-s|=R$.
Then
\[\Re\frac{f'}{f}(s)=\sum_{|\rho-s|<R}\Re\Big(\frac{1}{s-\rho}-\frac{s-\rho}{R^{2}}\Big)+\frac{1}{\pi R}\int_{0}^{2\pi}(\cos\alpha)\log|f(s+Re^{i\alpha})|\,d\alpha,\]
where $\rho$ ranges over zeros of $f$, counted with multiplicity.
\end{lem}

As an application of this formula, we deduce an upper bound for the logarithmic derivative that will serve as the main tool to obtain our zero-free region.

\begin{lem}\label{bound-L'/L}
Let $\chi\pmod{q}$ be a non-principal Dirichlet character.
Assuming \glh, then for any $\varepsilon>0$ there exists $\delta=\delta(\varepsilon)>0$ such that
\begin{equation}\label{eq-bound-L'/L}
-\Re\frac{L'}{L}(s,\chi)\leq-\sum_{|1+it-\rho|\leq\delta}\Re\frac{1}{s-\rho}+\varepsilon\L
\end{equation}
uniformly for
\begin{equation}\label{eq-range-sigma-t}
1+\frac{1}{\L\log\L}\leq\sigma\leq 1+\frac{\log\L}{\L},\qquad |t|\leq 2\L,
\end{equation}
provided $q\geq q_{1}(\varepsilon)$.
\end{lem}

\begin{proof}
Our proof follows that of \cite{Heath-Brown-Linnik}*{Lemma 3.1}, allowing us to obtain a sharper bound by leveraging GLH.
We present the argument here to keep the exposition self-contained.

Fix $\varepsilon>0$ and let $\varepsilon_{0}>0$ be a parameter to be chosen later, only depending on $\varepsilon$.
For $s$ in the region \eqref{eq-range-sigma-t} and $1/3<R\leq 1/2$ such that $L(z,\chi)$ has no zeros on the circle $|z-s|=R$, \cref{HB-jensen} yields
\begin{equation}\label{-reL'/L-equal}
-\Re\frac{L'}{L}(s,\chi)=-\sum_{|s-\rho|<R}\Re\Big(\frac{1}{s-\rho}-\frac{s-\rho}{R^{2}}\Big)-\frac{1}{\pi R}\int_{0}^{2\pi}(\cos\alpha)\log|L(s+Re^{i\alpha},\chi)|\,d\alpha.
\end{equation}
Next, we estimate the integral above.
By \eqref{eq-range-sigma-t}, in the ranges $0\leq\alpha\leq\pi/2$ and $3\pi/2\leq\alpha\leq 2\pi$ we have the trivial bound
\[|\log L(s+Re^{i\alpha},\chi)|\leq\log\zeta(\sigma+R\cos\alpha)\leq\log\zeta(\sigma)\ll\log\L,\]
thus the total contribution of these ranges to the integral is $\leq\varepsilon_{0}\L$, say.
On the other hand, for $\pi/2\leq\alpha\leq 3\pi/2$ we have $1/2\leq\sigma+R\cos(\alpha)\leq 1$, so \glh~\! implies
\[\log|L(s+Re^{i\alpha},\chi)|\leq 2\varepsilon_{0}\L\]
for $q$ large enough.
Therefore
\[\int_{\pi/2}^{3\pi/2}(\cos\alpha)\log|L(s+Re^{i\alpha},\chi)|\,d\alpha\geq 2\varepsilon_{0}\L\int_{\pi/2}^{3\pi/2}\cos\alpha\,d\alpha=-4\varepsilon_{0}\L.\]
Combining the contribution of each range, \eqref{-reL'/L-equal} gives
\begin{equation}\label{-reL'/L-leq1}
-\Re\frac{L'}{L}(s,\chi)\leq -\sum_{|s-\rho|<R}\Re\Big(\frac{1}{s-\rho}-\frac{s-\rho}{R^{2}}\Big)+5\varepsilon_{0}\L.
\end{equation}
Now consider $0<\delta<R-\L^{-1}\log\L$ a parameter to be chosen later, so that all the zeros $\rho$ with $|1+it-\rho|\leq\delta$ are included in \eqref{-reL'/L-leq1}.
Note that
\[\Re\Big(\frac{1}{s-\rho}-\frac{s-\rho}{R^{2}}\Big)=(\sigma-\beta)\Big(\frac{1}{|s-\rho|^{2}}-\frac{1}{R^{2}}\Big)>0\]
for each summand in \eqref{-reL'/L-leq1}, so we may discard all zeros that are not in the smaller disk to obtain
\begin{equation}\label{-reL'/L-leq2}
-\Re\frac{L'}{L}(s,\chi)\leq -\sum_{|1+it-\rho|\leq\delta}\Re\Big(\frac{1}{s-\rho}-\frac{s-\rho}{R^{2}}\Big)+5\varepsilon_{0}\L.
\end{equation}
On the other hand, if $|1+it-\rho|\leq\delta$ then \eqref{eq-range-sigma-t} and the triangle inequality imply
\[\Re\Big(\frac{s-\rho}{R^{2}}\Big)\leq\frac{\sigma-\beta}{R^{2}}\leq\frac{1}{R^{2}}\Big(\frac{\log\L}{\L}+\delta\Big),\]
and the sum in \eqref{-reL'/L-leq2} has $\leq \log\L$ terms by \eqref{eq-range-sigma-t} and the Riemann-von Mangoldt formula.
Thus
\begin{align*}
-\Re\frac{L'}{L}(s,\chi)&\leq -\sum_{|1+it-\rho|\leq\delta}\Re\frac{1}{s-\rho}+\Big(\frac{\log\L}{\L}+\delta\Big)\frac{\log\L}{R^{2}}+5\varepsilon_{0}\L\\
&\leq -\sum_{|1+it-\rho|\leq\delta}\Re\frac{1}{s-\rho}+(6\varepsilon_{0}+\delta)\L
\end{align*}
for $q$ large enough.
The result now follows by taking $\varepsilon_{0}=\varepsilon/7$ and $\delta=\min(1/4,\varepsilon/7)$.
\end{proof}

We now follow a classical approach (see \cite{Davenport-book}*{\S 14}), using \Cref{bound-L'/L} in place of the partial fraction decomposition of $L'/L$, to obtain a zero-free region for a Dirichlet $L$-function $L(s,\chi)$.

\begin{prop}\label{zero-free-LH}
Assume \glh.
For any $\eta\geq 1$, if $q\geq q_{2}(\eta)$, then $\prod_{\chi\mmod{q}}L(s,\chi)$ has at most one zero in the region
\begin{equation}\label{eq-region-zerofree}
\sigma\geq 1-\frac{\eta}{\L},\qquad |t|\leq\L.
\end{equation}
Such a zero, if it exists, is real, simple, and corresponds to a non-principal real character.
\end{prop}

\begin{proof}
We only consider $\chi$ non-principal, as in that case the result follows from known zero-free regions for the Riemann zeta function.
For fixed $\eta\geq 1$, let $\varepsilon=(100\eta)^{-1}$ and take $\delta=\delta(\varepsilon)$ and $q\geq q_{1}(\varepsilon)$ as in \Cref{bound-L'/L}.
We assume there exists a zero $\rho_{0}=\beta_{0}+i\gamma_{0}$ of $L(s,\chi)$ with $1-\beta_{0}\leq\delta$ and $|\gamma_{0}|\leq\L$, as otherwise the statement is immediate for $q$ large enough.
For any zero $\rho=\beta+i\gamma$ of $L(s,\chi)$ and $s=\sigma+it$ with $\sigma>1$, we have
\[\Re\frac{1}{s-\rho}=\frac{\sigma-\beta}{(\sigma-\beta)^{2}+(t-\gamma)^{2}}>0.\]
Thus, applying \Cref{bound-L'/L} to $L(z,\chi)$ and discarding all but the zero $\rho_{0}$ in \eqref{eq-bound-L'/L} we get
\begin{equation}\label{eq-zerofree-igamma}
-\Re\frac{L'}{L}(\sigma+i\gamma_{0},\chi)\leq-\frac{1}{\sigma-\beta_{0}}+\varepsilon\L.
\end{equation}
We would now like to apply \Cref{bound-L'/L} to $L(z,\chi^{2})$, but this is not possible if $\chi$ is real.
We therefore argue separately according to whether $\chi$ is complex or real.\vspace{1ex}

\noindent$\bullet$ $\chi$ \textbf{complex.}
Using \Cref{bound-L'/L} with $L(z,\chi^{2})$ and omitting all zeros in \eqref{eq-bound-L'/L} yields
\begin{equation}\label{eq-zerofree-2igamma}
-\Re\frac{L'}{L}(\sigma+2i\gamma_{0},\chi^{2})\leq\varepsilon\L.
\end{equation}
Combining \eqref{eq-zerofree-igamma} and \eqref{eq-zerofree-2igamma} with the bound
\begin{equation}\label{eq-zerofree-riemann}
-\frac{\zeta'}{\zeta}(\sigma)=\frac{1}{\sigma-1}+O(1)
\end{equation}
and the trigonometric inequality $3+4\cos\theta+\cos 2\theta\geq 0$, we obtain
\begin{align*}
\frac{3}{\sigma-1}-\frac{4}{\sigma-\beta_{0}}+5\varepsilon\L+O(1)&=-3\frac{\zeta'}{\zeta}(\sigma)-4\Re\frac{L'}{L}(\sigma+i\gamma_{0},\chi)-\Re\frac{L'}{L}(\sigma+2i\gamma_{0},\chi^{2})\\
&=\sum_{n=1}^{\infty}\frac{\Lambda(n)}{n^{\sigma}}\Big(3+4\Re\frac{\chi(n)}{n^{i\gamma_{0}}}+\Re\frac{\chi^{2}(n)}{n^{2i\gamma_{0}}}\Big)\geq 0.
\end{align*}
We now write $\sigma=1+y/\L$ for some constant $y>0$, so the inequality above becomes
\[\beta_{0}\leq 1-\frac{y}{\L}\cdot\frac{1-5\varepsilon y+O(\L^{-1})}{3+5\varepsilon y+O(\L^{-1})},\]
and taking $y=(10\varepsilon)^{-1}=10\eta$ and $q$ large enough we get
\[\beta_{0}\leq 1-\frac{y}{8\L}<1-\frac{\eta}{\L}.\]
\noindent$\bullet$ $\chi$ \textbf{real.}
Since $\chi^{2}$ is principal, we instead recall that, for $\sigma\geq 1$, we have
\begin{align*}
\Big|\frac{L'}{L}(s,\chi^{2})-\frac{\zeta'}{\zeta}(s)\Big|=\sum_{p\mid q}\frac{\log p}{|p^{s}-1|}\leq 2\sum_{p\mid q}\frac{\log p}{p}&=2\sum_{\substack{p\mid q\\ p<\L}}\frac{\log p}{p}+2\sum_{\substack{p\mid q\\ p\geq\L}}\frac{\log p}{p}\\
&\leq 2\sum_{p<\L}\frac{\log p}{p}+\frac{2}{\L}\sum_{p\mid q}\log p\ll\log\L.
\end{align*}
Moreover, since $|\gamma_{0}|\leq\L$ and
\[-\Re\frac{\zeta'}{\zeta}(s)\leq\Re\frac{1}{s-1}+O(\log(2+|t|)),\]
we obtain the coarser inequality
\begin{equation}\label{eq-zerofree-2igamma-real}
-\Re\frac{L'}{L}(\sigma+2i\gamma_{0},\chi^{2})\leq\Re\frac{1}{\sigma-1+2i\gamma_{0}}+O(\log\L).
\end{equation}
Repeating the argument as in the complex case, with \eqref{eq-zerofree-2igamma-real} in place of \eqref{eq-zerofree-2igamma}, yields
\[\frac{3}{\sigma-1}-\frac{4}{\sigma-\beta_{0}}+\frac{\sigma-1}{(\sigma-1)^{2}+4\gamma_{0}^{2}}+4\varepsilon\L+O(\log\L)\geq 0.\]
As before, we write $\sigma=1+y/\L$ and this time we additionally assume $|\gamma_{0}|\geq y/\L$, thus the inequality above implies
\[\beta_{0}\leq 1-\frac{y}{\L}\cdot\frac{1-5\varepsilon y+O(\L^{-1}\log\L)}{4+5\varepsilon y+O(\L^{-1}\log\L)},\]
and again taking $y=(10\varepsilon)^{-1}=10\eta$ and $q$ large enough we obtain $\beta_{0}<1-\eta/\L$.
Now it only remains to consider the case when $\chi$ is real and $|\gamma_{0}|<10\eta/\L$.
Assume first $\gamma_{0}\neq 0$, so that $\beta_{0}-i\gamma_{0}$ is also a zero of $L(s,\chi)$.
In this case, using \Cref{bound-L'/L} with $L(z,\chi)$ and discarding all but the zeros $\beta_{0}\pm i\gamma_{0}$ in \eqref{eq-bound-L'/L}, we obtain
\[-\frac{L'}{L}(\sigma,\chi)\leq -\frac{1}{\sigma-\beta_{0}-i\gamma_{0}}-\frac{1}{\sigma-\beta_{0}+i\gamma_{0}}+\varepsilon\L=-\frac{2(\sigma-\beta_{0})}{(\sigma-\beta_{0})^{2}+\gamma_{0}^{2}}+\varepsilon\L,\]
which together with the trivial bound
\[-\frac{L'}{L}(\sigma,\chi)=\sum_{n=1}^{\infty}\frac{\Lambda(n)\chi(n)}{n^{\sigma}}\geq -\sum_{n=1}^{\infty}\frac{\Lambda(n)}{n^{\sigma}}=\frac{\zeta'}{\zeta}(\sigma)=-\frac{1}{\sigma-1}+O(1)\]
gives
\begin{equation}\label{eq-zerofree-real1}
\frac{1}{\sigma-1}-\frac{2(\sigma-\beta_{0})}{(\sigma-\beta_{0})^{2}+\gamma_{0}^{2}}+\varepsilon\L+O(1)\geq 0.
\end{equation}
We now choose $\sigma=1+20\eta/\L$, and since $|\gamma_{0}|<10\eta/\L=(\sigma-1)/2<(\sigma-\beta_{0})/2$, \eqref{eq-zerofree-real1} implies
\[\beta_{0}\leq 1-\frac{20\eta}{\L}\cdot\frac{3-100\varepsilon\eta+O(\L^{-1})}{5+100\varepsilon\eta+O(\L^{-1})}<1-\frac{\eta}{\L}.\]
Lastly, suppose $L(z,\chi)$ has two (possibly equal, in which case we mean it has multiplicity two) real zeros $\beta_{1},\beta_{2}$ with $1-\beta_{j}\leq\delta$, as otherwise the statement is immediate for $q$ large, as before.
Proceeding as above, we get
\[-\frac{1}{\sigma-1}+O(1)\leq\frac{L'}{L}(\sigma,\chi)\leq -\frac{1}{\sigma-\beta_{1}}-\frac{1}{\sigma-\beta_{2}}+\varepsilon\L\leq -\frac{2}{\sigma-\min(\beta_{1},\beta_{2})}+\varepsilon\L,\]
and taking $\sigma=1+2\eta/\L$ we conclude
\begin{equation}\label{eq-two-real}
\min(\beta_{1},\beta_{2})\leq 1-\frac{2\eta}{\L}\cdot\frac{1-2\varepsilon\eta+O(\L^{-1})}{1+2\varepsilon\eta+O(\L^{-1})}<1-\frac{\eta}{\L}.
\end{equation}
In summary, we have shown that, given a non-principal character $\chi$, $L(z,\chi)$ has no zeros in the region \eqref{eq-region-zerofree} when $\chi$ is complex, and can have at most one zero in such region when $\chi$ is real.
Moreover, if such a zero exists, it must be real and simple by \eqref{eq-two-real}.

Hence, to complete the proof, it remains to show that, for a given $q$, there exists at most one such character.
Suppose $\chi_{1}$ and $\chi_{2}$ are two different non-principal real characters modulo $q$ such that the functions $L(z,\chi_{1})$ and $L(z,\chi_{2})$ have two real zeros $\beta_{1}$, $\beta_{2}$, respectively, with $1-\beta_{j}\leq\delta$.
Then $\chi_{1}\chi_{2}$ is a non-principal character modulo $q$, thus applying \Cref{bound-L'/L} to $L(z,\chi_{1}\chi_{2})$ and discarding all zeros in \eqref{eq-bound-L'/L} yields
\begin{equation}\label{eq-two-chi}
-\frac{L'}{L}(\sigma,\chi_{1}\chi_{2})\leq\varepsilon\L.
\end{equation}
Similarly, we use \Cref{bound-L'/L} with each $L(z,\chi_{j})$, omitting all but the zero $\beta_{j}$ in \eqref{eq-bound-L'/L}, to obtain
\begin{equation}\label{eq-each-chi}
-\frac{L'}{L}(\sigma,\chi_{j})\leq -\frac{1}{\sigma-\beta_{j}}+\varepsilon\L.
\end{equation} 
Finally, from \eqref{eq-zerofree-riemann}, \eqref{eq-two-chi}, and \eqref{eq-each-chi} we deduce
\begin{align*}
\frac{1}{\sigma-1}-\frac{1}{\sigma-\beta_{1}}-\frac{1}{\sigma-\beta_{2}}+3\varepsilon\L+O(1)&\geq -\frac{\zeta'}{\zeta}(\sigma)-\frac{L'}{L}(\sigma,\chi_{1})-\frac{L'}{L}(\sigma,\chi_{2})-\frac{L'}{L}(\sigma,\chi_{1}\chi_{2})\\
&=\sum_{n=1}^{\infty}\frac{\Lambda(n)}{n^{\sigma}}\big(1+\chi_{1}(n)\big)\big(1+\chi_{2}(n)\big)\geq 0,
\end{align*}
thus
\[\frac{2}{\sigma-\min(\beta_{1},\beta_{2})}\leq \frac{1}{\sigma-\beta_{1}}+\frac{1}{\sigma-\beta_{2}}\leq\frac{1}{\sigma-1}+3\varepsilon\L+O(1).\]
Taking $\sigma=1+2\eta/\L$ we get
\[\min(\beta_{1},\beta_{2})\leq 1-\frac{2\eta}{\L}\cdot\frac{1-6\varepsilon\eta+O(\L^{-1})}{1+6\varepsilon\eta+O(\L^{-1})}<1-\frac{\eta}{\L},\]
which concludes the proof.
\end{proof}

\subsection{Two zero-density estimates}

In this section we establish two zero-density estimates that will be used later: a log-free estimate valid in the strip $1/2\leq\sigma\leq 1$, and a much sharper estimate for the number of characters whose $L$-function has a zero very close to the line $\sigma=1$.\vspace{1ex}

Our first result is the assertion that \glh~\!\! implies the density hypothesis for Dirichlet $L$-functions.
This appears to be folklore, but we were unable to find a proof in the literature, so we provide one for the sake of completeness.
The proof relies on the following estimate for Dirichlet polynomials.

\begin{lem}\label{dirichlet-poly}
Let $k,m\geq 1$ be integers, $N,Q,T\geq 1$, $V>0$, and $1/2\leq\alpha\leq 1$.
Let $a_{n}$ be complex numbers, and define the Dirichlet polynomial
\[D(s,\chi)=\sum_{N<n\leq 2N}\frac{a_{n}\chi(n)}{n^{s}}.\]
Let $\mathcal{C}=\mathcal{C}(k,Q)$ be the set of Dirichlet characters $\chi=\psi\xi$, where $\psi$ is a primitive character modulo $q$ for some $q\leq Q$ with $(q,k)=1$, and $\xi$ is a character modulo $k$.
Suppose $\mathcal{S}$ is a set of pairs $(\rho,\chi)$ such that $\chi\in\mathcal{C}$ and $\rho=\beta+i\gamma$ is a zero of $L(s,\chi)$ with $\beta\geq\alpha$ and $|\gamma|\leq T$.
Assume further that $\mathcal{S}$ satisfies the following conditions:
\begin{itemize}
\item For every pair $(\rho,\chi)\in\mathcal{S}$, we have $|D(\rho,\chi)|\geq V$.
\item If $(\rho_{1},\chi_{1}),(\rho_{2},\chi_{2})\in\mathcal{S}$ are distinct, then either $\chi_{1}\neq\chi_{2}$, or $\chi_{1}=\chi_{2}$ and $|\gamma_{1}-\gamma_{2}|\geq 1$.
\end{itemize}
Then
\[|\mathcal{S}|\ll (H+N^{m})\big(G_{m}(\alpha)V^{-2}\big)^{m}(\log HN^{m})^{5},\]
where $H=kQ^{2}T$, $\tau_{m}$ is the $m$-fold divisor function, and
\[G_{m}(\alpha)=\sum_{N<n\leq 2N}\frac{|a_{n}|^{2}\tau_{m}(n)}{n^{2\alpha}}.\]
\end{lem}

\begin{proof}
This is precisely \cite{Iwaniec-Kowalski}*{Theorem 9.16} applied to the polynomial $D(s,\chi)^{m}$, where we used that $\tau_{m}(n_{1}\cdots n_{m})\leq\tau_{m}(n_{1})\cdots\tau_{m}(n_{m})$.
\end{proof}

With this preliminary lemma in place, we proceed to establish our first zero-density estimate, which in turn implies the log-free estimate required for the proof of \cref{main-thm}.

\begin{prop}\label{glh-implies-density}
Let $k,Q\geq 1$, with $k$ an integer.
For $\chi\pmod{\ell}$ a Dirichlet character, let $N(\alpha,T,\chi)$ denote the number of zeros $\rho=\beta+i\gamma$ of $L(s,\chi)$ with $\beta\geq\alpha$ and $|\gamma|\leq T$, counted with multiplicity.
Define
\[N(\alpha,k,Q,T)=\sum_{\substack{q\leq Q\\ (q,k)=1}}\sum_{\substack{\psi\mmod{q}\\ \psi~\!\mathrm{ primitive}}}\sum_{\xi\mmod{k}}N(\alpha,T,\xi\psi).\]
Assuming \glh, then for all $\varepsilon>0$ we have
\[N(\alpha,k,Q,T)\ll_{\varepsilon}(kQ^{2}T)^{2(1-\alpha)+\varepsilon}\]
for $1/2\leq\alpha\leq 1$.
\end{prop}

\begin{cor}\label{log-free-density}
Assuming \glh, then for all $\varepsilon>0$ we have
\begin{equation}\label{eq-log-free-density}
\sum_{\chi\mmod{q}}N(\alpha,T,\chi)\ll_{\varepsilon}(qT)^{(2+\varepsilon)(1-\alpha)}
\end{equation}
for $1/2\leq\alpha\leq 1$.
\end{cor}

\begin{proof}[Proof of \cref{log-free-density} assuming \cref{glh-implies-density}]
In the range $4/5\leq\alpha\leq 1$ this is precisely the log-free zero-density estimate of Jutila \cite{Jutila-Linnik-1977}*{Theorem 1}, while for the remaining range $1/2\leq\alpha\leq 4/5$ the result follows from setting $Q=1$ in \cref{glh-implies-density}.
\end{proof}

\begin{proof}[Proof of \cref{glh-implies-density}]
We roughly follow \cite{Iwaniec-Kowalski}*{\S 10.3}, and to simplify the notation we write $H=kQ^{2}T$.
First, note that there is exactly one principal character contributing to our sum $N(\alpha,k,Q,T)$.
Denoting this character by $\chi_{0}$, the zeros of $L(s,\chi_{0})$ in the region $\sigma\geq 1/2$ coincide with those of $\zeta(s)$.
Under \glh, which includes the Lindel\"of hypothesis for $\zeta(s)$, it is a well known theorem of Ingham \cite{Ingham-Lindelof}*{Theorem 3} that $N(\alpha,T,\zeta)\ll_{\varepsilon}T^{2(1-\alpha)+\varepsilon}$.
As $T\leq H$, this contribution is $\ll_{\varepsilon}H^{2(1-\alpha)+\varepsilon}$, which will be absorbed into our final upper bound.
Therefore, we may restrict our attention to non-principal characters $\chi$ modulo $\ell=kq$.

Let $2\leq X\leq H$, $\delta=1/\log H$, $s=\sigma+it$, with $1/2\leq\sigma\leq 1$ and $|t|\leq T$, and fix $0<\varepsilon<0.01$.
By Perron's formula we have
\[\sum_{n\leq X}\frac{\chi(n)}{n^{s}}=\frac{1}{2\pi i}\int_{1-\sigma+\delta-iH}^{1-\sigma+\delta+iH}L(s+w,\chi)\frac{X^{w}}{w}\,dw+R,\]
where
\[R\ll X^{-\sigma}+\frac{X^{1+\delta-\sigma}\log X}{H}\ll X^{1/2-\sigma}.\]
Assuming \glh, we have $L(s+w,\chi)\ll_{\varepsilon}H^{\varepsilon/4}$ on our contour, while on the segment $\Re(w)=1/2-\sigma-\delta$, $|\Im(w)|\leq H$, the functional equation for the character inducing $\chi$ combined with \glh~\! yields the same bound.
Therefore, by the Phragmen--Lindel\"of principle this bound extends to the entire rectangle $1/2-\sigma-\delta\leq\Re(w)\leq 1-\sigma+\delta$, $|\Im(w)|\leq H$.
With this uniform bound, we shift the contour of integration to the line $\Re(w)=1/2-\sigma-\delta$, and we pick up a simple pole at $w=0$ with residue $L(s,\chi)$.
Moreover, the contribution from the horizontal integrals is
\[\ll\int_{1/2-\sigma-\delta}^{1-\sigma+\delta}\frac{H^{\varepsilon/4}X^{u}}{H}\,du\ll\frac{X^{1-\sigma+\delta}H^{\varepsilon/4}}{H}\ll X^{1/2-\sigma},\]
and similarly the integral over the left edge $\Re(w)=1/2-\sigma-\delta$ is
\[\ll\int_{-H}^{H}\frac{H^{\varepsilon/4}X^{1/2-\sigma-\delta}}{|1/2-\sigma-\delta+iv|}\,dv\ll X^{1/2-\sigma-\delta}H^{\varepsilon/4}\log H\ll_{\varepsilon}X^{1/2-\sigma}H^{\varepsilon/2}.\]
In conclusion, we obtain
\[L(s,\chi)=\sum_{n\leq X}\frac{\chi(n)}{n^{s}}+O_{\varepsilon}\big(X^{1/2-\sigma}H^{\varepsilon/2}\big).\]
Consider now a parameter $Y\geq 2$ and define the mollifier
\[M(s,\chi)=\sum_{n\leq Y}\frac{\mu(n)\chi(n)}{n^{s}}.\]	
Using the trivial bound $M(s,\chi)\ll Y^{1-\sigma}\log Y$ we obtain
\[L(s,\chi)M(s,\chi)=\sum_{n\leq XY}\frac{a_{n}\chi(n)}{n^{s}}+O_{\varepsilon}\big(X^{1/2-\sigma}Y^{1-\sigma}H^{\varepsilon/2}\log Y\big),\]
where
\begin{equation}\label{def-an}
a_{n}=\sum_{\substack{d\mid n\\ n/X\leq d\leq Y}}\mu(d),\qquad |a_{n}|\leq\tau(n).
\end{equation}
Note that $a_{1}=1$, and assuming $Y\leq X$ we have $a_{n}=\sum_{d\mid n}\mu(d)=0$ for $1<n\leq Y$.
We now decompose the interval $Y<n\leq XY$ into dyadic subintervals $N_{j}<n\leq 2N_{j}$, with $N_{j}=2^{j}Y$, $0\leq j<J=\log X/\log 2$.
For each $j$ we denote
\[D_{j}(s,\chi)=\sum_{N_{j}<n\leq 2N_{j}}\frac{a_{n}\chi(n)}{n^{s}},\]
thus
\begin{equation}\label{zero-detector}
L(s,\chi)M(s,\chi)=1+\sum_{0\leq j<J}D_{j}(s,\chi)+O_{\varepsilon}\big(X^{1/2-\sigma}Y^{1-\sigma}H^{\varepsilon/2}\log Y\big).
\end{equation}
We now require $X^{\sigma-1/2}\geq Y^{1-\sigma}H^{\varepsilon}$, so that if $\rho$ is a zero of $L(s,\chi)$ in the region under consideration, and $H$ is large enough, then \eqref{zero-detector} yields
\[\Big|\sum_{0\leq j<J}D_{j}(\rho,\chi)\Big|\geq\frac{1}{2}.\]
In particular, this implies $|D_{j}(\rho,\chi)|\geq (2J)^{-1}$ for some $0\leq j<J$.

With the notation of \cref{dirichlet-poly}, let $\mathcal{Z}_{j}$ be the multiset of pairs $(\rho,\chi)$ such that $\chi\in\mathcal{C}$, $\rho=\beta+i\gamma$ is a zero of $L(s,\chi)$ with $\beta\geq\alpha$ and $|\gamma|\leq T$, counted with multiplicity, and $|D_{j}(\rho,\chi)|\geq(2J)^{-1}$.
From $\mathcal{Z}_{j}$ we extract a subset $\mathcal{S}_{j}$ that satisfies the conditions of \cref{dirichlet-poly} as follows.
For each $\chi\in\mathcal{C}$, consider the zeros $\rho=\beta+i\gamma$ such that $(\rho,\chi)\in\mathcal{Z}_{j}$, and choose a maximal subset $\mathcal{S}_{j}(\chi)$ with the property that any two distinct zeros $\rho_{1},\rho_{2}$ in this subset satisfy $|\gamma_{1}-\gamma_{2}|\geq 1$.
We then define $\mathcal{S}_{j}=\bigcup_{\chi\in\mathcal{C}}\{(\rho,\chi):\rho\in\mathcal{S}_{j}(\chi)\}$.
Because our selection for each character is maximal, for every $(\rho,\chi)\in\mathcal{Z}_{j}$ there exists $(\rho',\chi)\in\mathcal{S}_{j}$ with $|\gamma-\gamma'|\leq 1$.
Moreover, by standard estimates the function $L(s,\chi)$ has $\ll\log(kQT)$ zeros, counted with multiplicity, whose imaginary part $\gamma$ satisfy $|\gamma-\gamma'|\leq 1$.
Thus $|\mathcal{Z}_{j}|\ll |\mathcal{S}_{j}|\log H$.
Recall also that the contribution to $N(\alpha,k,Q,T)$ from the principal character is $\ll_{\varepsilon}H^{2(1-\alpha)+\varepsilon}$, and the zeros of the remaining characters are all contained in some $\mathcal{Z}_{j}$.
Therefore
\begin{equation}\label{final-sum-1}
N(\alpha,k,Q,T)\ll_{\varepsilon}H^{2(1-\alpha)+\varepsilon}+\sum_{0\leq j<J}|\mathcal{Z}_{j}|\ll H^{2(1-\alpha)+\varepsilon}+J\log H\max_{0\leq j<J}|\mathcal{S}_{j}|.
\end{equation}
To bound $|\mathcal{S}_{j}|$ using \cref{dirichlet-poly}, we will raise the polynomials $D_{j}(s,\chi)$ to appropriate powers to optimize their lengths.
We do this by introducing a parameter $Z$, with $(XY)^{1/\sqrt{\varepsilon}}\leq Z\leq H$, to be chosen later.
As $2\leq N_{j}\leq XY\leq Z^{\sqrt{\varepsilon}}$, for each $0\leq j<J$ let $m_{j}$ be the unique integer satisfying $N_{j}^{m_{j}-1}\leq Z<N_{j}^{m_{j}}$, so that
\begin{equation}\label{ineq-N^m}
Z<N_{j}^{m_{j}}=N_{j}^{m_{j}-1}N_{j}\leq Z^{1+\sqrt{\varepsilon}}.
\end{equation}
To simplify the notation, from now on we write $P_{j}=N_{j}^{m_{j}}$.
Using \eqref{def-an} and the estimate
\[\sum_{N<n\leq 2N}\tau(n)^{2}\tau_{m}(n)\ll N(\log N)^{4m-1}\]
we obtain
\[\Big(\sum_{N_{j}<n\leq 2N_{j}}\frac{|a_{n}|^{2}\tau_{m_{j}}(n)}{n^{2\alpha}}\Big)^{m_{j}}\ll P_{j}^{1-2\alpha}(\log N_{j})^{m_{j}(4m_{j}-1)}.\]
Thus \cref{dirichlet-poly} implies
\[|\mathcal{S}_{j}|\ll(H+P_{j})P_{j}^{1-2\alpha}(\log H)^{A_{j}},\]
where $A_{j}=m_{j}(4m_{j}+1)+5$.
Applying the bounds for $P_{j}$ from \eqref{ineq-N^m}, and recalling that $1/2\leq\alpha\leq 1$, we obtain
\[|\mathcal{S}_{j}|\ll\big(HZ^{1-2\alpha}+Z^{2(1+\sqrt{\varepsilon})(1-\alpha)}\big)(\log H)^{A_{j}}.\]
Choosing $Z=H^{1/(1+2\sqrt{\varepsilon}(1-\alpha))}$ to minimize this expression gives
\begin{equation}\label{opt-sj-bound}
|\mathcal{S}_{j}|\ll H^{2(1-\alpha)\big(1+\frac{\sqrt{\varepsilon}(2\alpha-1)}{1+2\sqrt{\varepsilon}(1-\alpha)}\big)}(\log H)^{A_{j}},
\end{equation}
which is valid provided we can choose parameters $X$ and $Y$ satisfying
\[2\leq Y\leq X\leq H,\qquad XY\leq H^{\sqrt{\varepsilon}/(1+2\sqrt{\varepsilon}(1-\alpha))},\qquad X^{\alpha-1/2}\geq Y^{1-\alpha}H^{\varepsilon}.\]
Indeed, for $1/2+2\sqrt{\varepsilon}\leq\alpha\leq 1$, a straightforward verification shows that these conditions are satisfied by setting $X=H^{x}$ and $Y=H^{y}$, where
\[x=\frac{\sqrt{\varepsilon}}{1+2\sqrt{\varepsilon}(1-\alpha)}-\varepsilon,\qquad y=\varepsilon.\]
With these choices, the definition of $m_{j}$ implies
\[m_{j}=1+\Big\lfloor\frac{\log Z}{\log N_{j}}\Big\rfloor\leq 1+\frac{\log Z}{\log Y}=1+\frac{1}{\varepsilon(1+2\sqrt{\varepsilon}(1-\alpha))}\ll\varepsilon^{-1},\]
thus $A_{j}\ll\varepsilon^{-2}$.
We now combine this bound with \eqref{final-sum-1} and \eqref{opt-sj-bound}.
As $\varepsilon$ is sufficiently small, the term $H^{2(1-\alpha)+\varepsilon}$ in \eqref{final-sum-1} is absorbed by the second, yielding
\[N(\alpha,k,Q,T)\ll_{\varepsilon}H^{2(1-\alpha)\big(1+\frac{\sqrt{\varepsilon}(2\alpha-1)}{1+2\sqrt{\varepsilon}(1-\alpha)}\big)}(\log H)^{O(\varepsilon^{-2})}\ll_{\varepsilon}H^{2(1-\alpha)+5\sqrt{\varepsilon}}\]
for $1/2+2\sqrt{\varepsilon}\leq\alpha\leq 1$.
It remains to consider the range $1/2\leq\alpha\leq 1/2+2\sqrt{\varepsilon}$.
Here, the Riemann-von Mangoldt formula gives the trivial bound
\[N(\alpha,k,Q,T)\ll\sum_{\substack{q\leq Q\\ (q,k)=1}}\sum_{\substack{\psi\mmod{q}\\ \psi~\!\mathrm{ primitive}}}\sum_{\xi\mmod{k}}T\log(kqT)\ll H\log H.\]
In this range we have $H\log H\ll_{\varepsilon}H^{1+\sqrt{\varepsilon}}\leq H^{2(1-\alpha)+5\sqrt{\varepsilon}}$, thus the bound
\[N(\alpha,k,Q,T)\ll_{\varepsilon}H^{2(1-\alpha)+5\sqrt{\varepsilon}}\]
holds for all $1/2\leq\alpha\leq 1$.
\end{proof}

For our purposes, \Cref{log-free-density} will be used to handle most of the zeros across the critical strip, but for the remaining zeros near the line $\Re(s)=1$ we need an estimate that is both explicit and sharper in this region, which is provided by the following proposition.

\begin{prop}\label{bound-Nlambda}
Let $N(\lambda)$ be the number of distinct Dirichlet characters $\chi$ modulo $q$ for which $L(s,\chi)$ has at least one zero in the region $\sigma\geq 1-\lambda/\L$, $|t|\leq 1$.
We label the characters $\chi^{(1)},\ldots,\chi^{(N)}$, where $N=N(\lambda)$, and for each $\chi^{(k)}$ we consider $\rho^{(k)}=1-\lambda^{(k)}/\L+i\gamma^{(k)}$ a corresponding zero.
Fix $\delta,a,b>0$ and assume \glh.
If $\lambda\leq\log\L$ and $q\geq q_{3}(\delta,a,b)$, then
\begin{equation}\label{N-lambda-general}
\sum_{k=1}^{N(\lambda)}\frac{\lambda^{(k)}}{e^{(3a+2b)\lambda^{(k)}}-e^{2a\lambda^{(k)}}}\leq\frac{a+b}{2ab}+\delta.
\end{equation}
In particular, taking $\delta=1$, $a=1/50$, and $b=a/\sqrt{2}$, we have that \glh~\! implies
\begin{equation}\label{N-lambda-particular}
\sum_{k=1}^{N(\lambda)}e^{-0.1\lambda^{(k)}}\leq 3
\end{equation}
for $\lambda\leq\log\L$ and $q$ large enough.
\end{prop}

\begin{proof}
As in \cref{bound-L'/L}, we adapt the proof of \cite{Heath-Brown-Linnik}*{Lemma 11.1} to the GLH setting and provide the details for the reader's convenience.

Let $U=q^{u}$, $V=q^{v}$, $W=q^{w}$, and $X=q^{x}$ be parameters, for some constants $v>u>0$ and $w,x>0$ to be specified later, and define
\begin{equation}\label{def-psi-theta}
\psi_{d}=\left\{\begin{array}{ll}
\mu(d) & \text{ if }1\leq d\leq U,\\
\mu(d)\frac{\log(V/d)}{\log(V/U)} & \text{ if }U\leq d\leq V,\\
0 & \text{ if }d\geq V,
\end{array}\right.\qquad\theta_{d}=\left\{\begin{array}{ll}
\mu(d)\frac{\log(W/d)}{\log W} & \text{ if }1\leq d\leq W,\\
0 & \text{ if }d\geq W.
\end{array}\right.
\end{equation}
To simplify the notation, let's temporarily write $\chi=\chi^{(k)}$ and $\rho=\rho^{(k)}$, with $\rho=\beta+i\gamma$.
By Mellin inversion we have
\begin{equation}\label{sum-psi-theta-1}
\sum_{n=1}^{\infty}\Big(\sum_{d\mid n}\psi_{d}\Big)\Big(\sum_{d\mid n}\theta_{d}\Big)\frac{\chi(n)}{n^{\rho}}e^{-n/X}=\frac{1}{2\pi i}\int_{2-i\infty}^{2+i\infty}G(s+\rho,\chi)L(s+\rho,\chi)\Gamma(s)X^{s}\,ds,
\end{equation}
where
\[G(s,\chi)=\sum_{a\leq V}\sum_{b\leq W}\psi_{a}\theta_{b}\frac{\chi([a,b])}{[a,b]^{s}},\]
and $[a,b]$ denotes the least common multiple of $a$ and $b$.
For $s$ in the strip $1/2-\beta\leq\sigma\leq 2$ we have
$\Gamma(s)\ll e^{-|t|}$,
\[G(s+\rho,\chi)\ll\sum_{a\leq V}\sum_{b\leq W}\frac{1}{[a,b]^{1/2}}\ll\sum_{n\leq VW}\frac{d(n)^{2}}{n^{1/2}}\ll (VW)^{1/2}\L^{2},\]
and $L(s+\rho,\chi)\ll_{\varepsilon}(q(1+|t|))^{\varepsilon}$ for any $\varepsilon>0$ by \glh.
These bounds allow us to shift the integral above to $\Re(s)=1/2-\beta$, and the new integral is
\[\int_{1/2-\beta-i\infty}^{1/2-\beta+i\infty}G(s+\rho,\chi)L(s+\rho,\chi)\Gamma(s)X^{s}\,ds\ll q^{\varepsilon}(VWX^{-1})^{1/2}\L^{2}X^{1-\beta}\ll q^{2\varepsilon+(v+w-x)/2},\]
where we used that $1-\beta\ll\L^{-1}\log\L$.
Therefore, for $x>v+w$ we choose $\varepsilon<(x-v-w)/4$, to obtain
\begin{equation}\label{sum-psi-theta-2}
\sum_{n=1}^{\infty}\Big(\sum_{d\mid n}\psi_{d}\Big)\Big(\sum_{d\mid n}\theta_{d}\Big)\frac{\chi(n)}{n^{\rho}}e^{-n/X}=O(\L^{-1}).
\end{equation}
Moreover, from \eqref{def-psi-theta} we see that $\sum_{d\mid n}\psi_{d}=0$ for $2\leq n\leq U$, thus
\begin{equation}\label{zero-det-1}
\begin{aligned}
&\sum_{n=1}^{\infty}\Big(\sum_{d\mid n}\psi_{d}\Big)\Big(\sum_{d\mid n}\theta_{d}\Big)\frac{\chi(n)}{n^{\rho}}e^{-n/X}\\
&=e^{-\L^{2}/U}+\sum_{n=1}^{\infty}\Big(\sum_{d\mid n}\psi_{d}\Big)\Big(\sum_{d\mid n}\theta_{d}\Big)\frac{\chi(n)}{n^{\rho}}(e^{-n/X}-e^{-n\L^{2}/U})+O\Big(\sum_{n>U}d(n)^{2}e^{-n\L^{2}/U}\Big).
\end{aligned}
\end{equation}
Summation by parts then yields
\[\sum_{n>U}d(n)^{2}e^{-n\L^{2}/U}\ll U(\log U)^{3}e^{-\L^{2}}+\frac{\L^{2}}{U}\int_{U}^{\infty}y\log(y)^{3}e^{-y\L^{2}/U}\,dy\ll\L^{-1},\]
and since $e^{-\L^{2}/U}=1+O(\L^{-1})$, \eqref{sum-psi-theta-2} and \eqref{zero-det-1} imply
\[1\leq\big(1+O(\L^{-1})\big)\Big|\sum_{n=1}^{\infty}\Big(\sum_{d\mid n}\psi_{d}\Big)\Big(\sum_{d\mid n}\theta_{d}\Big)\frac{\chi(n)}{n^{\rho}}(e^{-n/X}-e^{-n\L^{2}/U})\Big|^{2}.\]
For each $\chi=\chi^{(k)}$, let $w_{\chi}\geq 0$ be weights to be chosen later.
Multiplying the inequality above by $w_{\chi}$ and summing over $\chi$ we obtain
\begin{equation}\label{before-duality}
\sum_{\chi}w_{\chi}\leq\big(1+O(\L^{-1})\big)\sum_{\chi}\Big|\sum_{n=1}^{\infty}a_{n,\chi}b_{n}\Big|^{2},
\end{equation}
where
\[a_{n,\chi}=w_{\chi}^{1/2}\Big(\sum_{d\mid n}\theta_{d}\Big)\frac{\chi(n)}{n^{\rho-1/2}}(e^{-n/X}-e^{-n\L^{2}/U})^{1/2}\]
and
\[b_{n}=\Big(\sum_{d\mid n}\psi_{d}\Big)n^{-1/2}(e^{-n/X}-e^{-n\L^{2}/U})^{1/2}.\]
In this setting, it is now convenient to apply the following duality principle for bilinear forms: If $\sum_{n}|\sum_{\chi}A_{n,\chi}C_{\chi}|^{2}\leq M\sum_{\chi}|C_{\chi}|^{2}$ for all $C_{\chi}$, then $\sum_{\chi}|\sum_{n}A_{n,\chi}B_{n}|^{2}\leq M\sum_{n}|B_{n}|^{2}$ for all $B_{n}$.
Thus, in order to bound the right-hand side of \eqref{before-duality} we first aim for a bound of the form
\begin{equation}\label{duality-a}
\sum_{n=1}^{\infty}\Big|\sum_{\chi}a_{n,\chi}C_{\chi}\Big|^{2}\leq M\sum_{\chi}|C_{\chi}|^{2}
\end{equation}
for arbitrary $C_{\chi}$.
Expanding the left-hand side of \eqref{duality-a}, the non-diagonal terms are of the form
\[C_{\chi}\overline{C_{\chi'}}(w_{\chi}w_{\chi'})^{1/2}\sum_{n=1}^{\infty}\Big(\sum_{d\mid n}\theta_{d}\Big)^{2}\frac{\chi(n)\overline{\chi'}(n)}{n^{\rho+\rho'-1}}(e^{-n/X}-e^{-n\L^{2}/U}).\]
By essentially repeating the argument used to obtain \eqref{sum-psi-theta-2} from \eqref{sum-psi-theta-1}, we get
\[\sum_{n=1}^{\infty}\Big(\sum_{d\mid n}\theta_{d}\Big)^{2}\frac{\chi(n)\overline{\chi'}(n)}{n^{\rho+\overline{\rho'}-1}}(e^{-n/X}-e^{-n\L^{2}/U})\ll q^{\varepsilon}WU^{-1/2}\L^{3/2}(U/\L^{2})^{2(1-\beta)}\ll q^{2\varepsilon+w-u/2}.\]
Hence, for $u>2w$ we may choose $\varepsilon>0$ sufficiently small and conclude
\[\sum_{n=1}^{\infty}\Big(\sum_{d\mid n}\theta_{d}\Big)^{2}\frac{\chi(n)\overline{\chi'}(n)}{n^{\rho+\overline{\rho'}-1}}(e^{-n/X}-e^{-n\L^{2}/U})=O(\L^{-1}),\]
thus the contribution to \eqref{duality-a} from non-diagonal terms is
\begin{equation}\label{cont-off-diag}
\ll\L^{-1}\Big(\sum_{\chi}|C_{\chi}|w_{\chi}^{1/2}\Big)^{2}\ll\L^{-1}\Big(\sum_{\chi}|C_{\chi}|^{2}\Big)\Big(\sum_{\chi}w_{\chi}\Big).
\end{equation}
On the other hand, the diagonal terms are of the form
\[|C_{\chi}|^{2}w_{\chi}\sum_{n=1}^{\infty}\Big(\sum_{d\mid n}\theta_{d}\Big)^{2}\frac{\chi_{0}(n)}{n^{2\beta-1}}(e^{-n/X}-e^{-n\L^{2}/U}),\]
where $\chi_{0}$ is the principal character modulo $q$.
To bound these terms we will use the following estimate due to Graham \cite{Graham-Barban-Vehov}*{Corollary page 84}:
\begin{equation}\label{est-Theta}
\Theta(N):=\sum_{n\leq N}\Big(\sum_{d\mid n}\theta_{d}\Big)^{2}=\left\{\begin{array}{ll}
\frac{N\log N}{(\log W)^{2}}+O\big(\frac{N}{(\log W)^{2}}\big) & \text{if }1\leq N\leq W,\vspace{1ex}\\
\frac{N}{\log W}+O\big(\frac{N}{(\log W)^{2}}\big) & \text{if }N\geq W.
\end{array}\right.
\end{equation}
To simplify the notation, let $g(y)=y^{1-2\beta}(e^{-y/X}-e^{-y\L^{2}/U})$.
Summation by parts gives
\begin{equation}\label{theta-by-parts}
\sum_{n=1}^{\infty}\Big(\sum_{d\mid n}\theta_{d}\Big)^{2}n^{1-2\beta}(e^{-n/X}-e^{-n\L^{2}/U})=-\int_{1}^{\infty}\Theta(y)g'(y)\,dy,
\end{equation}
and we now split this integral and study each range separately.
First, using \eqref{est-Theta} and the fact that $x>u>w$ and $1-\beta\ll\L^{-1}\log\log\L$, for $1\leq y\leq W$ we have
\[-g'(y)=y^{-2\beta}\Big(e^{-y/X}\Big(1-2\beta-\frac{y}{X}\Big)-e^{-y\L^{2}/U}\Big(1-2\beta-\frac{y\L^{2}}{U}\Big)\Big)\ll\frac{y^{1-2\beta}\L^{2}}{U}\ll\frac{\L^{3}}{yU},\]
thus
\begin{equation}\label{Theta-leq-W}
-\int_{1}^{W}\Theta(y)g'(y)\,dy\ll\frac{\L^{3}}{U(\log W)^{2}}\int_{1}^{W}\log y\,dy\ll\frac{W\L^{2}}{U}.
\end{equation}
Similarly, \eqref{est-Theta} and integration by parts yield
\begin{align*}
-\int_{W}^{\infty}\Theta(y)&g'(y)\,dy=-\frac{1+O(\L^{-1})}{\log W}\int_{W}^{\infty}g'(y)y\,dy\\
&=\frac{1+O(\L^{-1})}{\log W}\Big(W^{2-2\beta}(e^{-W/X}-e^{-W\L^{2}/U})+\int_{W}^{\infty}y^{1-2\beta}(e^{-y/X}-e^{-y\L^{2}/U})\,dy\Big)\\
&=\frac{\Gamma(2-2\beta)}{\L w}\Big(X^{2-2\beta}-\Big(\frac{U}{\L^{2}}\Big)^{2-2\beta}\Big)\big(1+O(\L^{-1})\big).
\end{align*}
In order to simplify this expression we recall that $e^{z}-1\geq z$ for real $z$.
Thus
\begin{equation}\label{lower-bound-main}
X^{2-2\beta}-U^{2-2\beta}=U^{2-2\beta}\Big(\Big(\frac{X}{U}\Big)^{2-2\beta}-1\Big)\geq (2-2\beta)U^{2-2\beta}\log(X/U)\gg U^{2-2\beta}(1-\beta)\L.
\end{equation}
This estimate, together with the bound $1-\beta\ll\L^{-1}\log\L$ and the Taylor expansion $\L^{z}=1+O(z\log\L)$ at $z=0$, imply
\begin{align*}
X^{2-2\beta}-\Big(\frac{U}{\L^{2}}\Big)^{2-2\beta}&=X^{2-2\beta}-U^{2-2\beta}\big(1+O\big((1-\beta)\L\big)\big)\\
&=\big(X^{2-2\beta}-U^{2-2\beta}\big)\big(1+O(\L^{-1}\log\L)\big).
\end{align*}
Similarly, we have $\Gamma(2-2\beta)=(2-2\beta)^{-1}\big(1+O(\L^{-1}\log\L)\big)$, and therefore
\begin{equation}\label{Theta-geq-W}
-\int_{W}^{\infty}\Theta(y)g'(y)\,dy=\frac{X^{2-2\beta}-U^{2-2\beta}}{2(1-\beta)\L w}\big(1+O(\L^{-1}\log\L)\big).
\end{equation}
Lastly, from \eqref{Theta-leq-W}, \eqref{lower-bound-main}, and \eqref{Theta-geq-W}, we see that the contribution of the range $1\leq y\leq W$ to the integral in \eqref{theta-by-parts} is negligible.
Hence, by choosing
\begin{equation}\label{def-weights}
w_{\chi}=\frac{2(1-\beta)\L w}{X^{2-2\beta}-U^{2-2\beta}}
\end{equation}
we conclude that the contribution of the diagonal terms to \eqref{duality-a} is
\begin{equation}\label{cont-diag}
\big(1+O(\L^{-1}\log\L)\big)\sum_{\chi}|C_{\chi}|^{2}.
\end{equation}
Together, \eqref{cont-off-diag} and \eqref{cont-diag} show that we may take
\begin{equation}\label{choice-M}
M=1+O\big(\L^{-1}\sum_{\chi}w_{\chi}\big)+O\big(\L^{-1}\log\L\big)
\end{equation}
in \eqref{duality-a}.
In conclusion, applying the duality principle and \eqref{duality-a} in \eqref{before-duality}, we obtain
\begin{equation}\label{after-duality}
\sum_{\chi}w_{\chi}\leq M\big(1+O(\L^{-1})\big)\sum_{n=1}^{\infty}\Big(\sum_{d\mid n}\psi_{d}\Big)^{2}n^{-1}(e^{-n/X}-e^{-n\L^{2}/U}).
\end{equation}
Following the same approach used to evaluate \eqref{theta-by-parts} using \eqref{est-Theta}, in this case the corresponding estimate of Graham is
\[\sum_{n\leq N}\Big(\sum_{d\mid n}\psi_{d}\Big)^{2}=\left\{\begin{array}{ll}
1 & \text{ if }1\leq N\leq U,\vspace{1ex}\\
\frac{N\log(N/U)}{(\log(V/U))^{2}}+O\big(\frac{N}{(\log(V/U))^{2}}\big) & \text{ if }U\leq N\leq V,\vspace{1ex}\\
\frac{N}{\log(V/U)}+O\big(\frac{N}{(\log(V/U))^{2}}\big) & \text{ if }N\geq V,
\end{array}\right.\]
and yields
\begin{equation}\label{sum-psi}
\sum_{n=1}^{\infty}\Big(\sum_{d\mid n}\psi_{d}\Big)^{2}n^{-1}(e^{-n/X}-e^{-n\L^{2}/U})=\frac{2x-u-v}{2(v-u)}\big(1+O(\L^{-1})\big).
\end{equation}
In this case the significant contributions come from the ranges $U\leq n\leq V$ and $n\geq V$, which are $1/2+O(\L^{-1})$ and $(x-v)/(v-u)+O(\L^{-1})$, respectively.
Finally, from \eqref{choice-M}, \eqref{after-duality}, and \eqref{sum-psi}, we get
\begin{align*}
\sum_{\chi}w_{\chi}&\leq\frac{2x-u-v}{2(v-u)}\Big(1+O\big(\L^{-1}\sum_{\chi}w_{\chi}\big)+O\big(\L^{-1}\log\L\big)\Big)\big(1+O(\L^{-1})\big)\\
&=\frac{2x-u-v}{2(v-u)}\big(1+O(\L^{-1}\log\L)\big)+O\big(\L^{-1}\sum_{\chi}w_{\chi}\big),
\end{align*}
and therefore
\begin{equation}\label{final-bound-sum-chi}
\sum_{\chi}w_{\chi}\leq\frac{2x-u-v}{2(v-u)}\big(1+O(\L^{-1}\log\L)\big).
\end{equation}
From our choice of weights \eqref{def-weights}, by writing $\chi=\chi^{(k)}$ and $\rho=\rho^{(k)}=1-\lambda^{(k)}/\L+i\gamma^{(k)}$, we conclude
\[\frac{4w(v-u)}{2x-u-v}\sum_{k=1}^{N(\lambda)}\frac{\lambda^{(k)}}{e^{2x\lambda^{(k)}}-e^{2u\lambda^{(k)}}}\leq 1+O(\L^{-1}\log\L)\]
provided $v>u>2w>0$ and $x>v+w$.
In order to minimize the resulting bound, we choose the parameters subject to the above constraints so as to maximize
\[\frac{w(v-u)}{(2x-u-v)(e^{2x\lambda^{(k)}}-e^{2u\lambda^{(k)}})}.\]
By monotonicity considerations, this is achieved by taking $x$ and $u$ as small as permitted, which under our constraints leads to the choices $x=v+w+\delta$ and $u=2w+\delta$, with $\delta>0$ small.
Writing $w=a/2$ and $v=u+b$ and substituting into \eqref{final-bound-sum-chi} yields \eqref{N-lambda-general}, with a new value of $\delta$.
The bound \eqref{N-lambda-particular} then follows from \eqref{N-lambda-general} and the inequality $e^{C\Lambda}-e^{D\Lambda}\leq (C-D)\Lambda e^{C\Lambda}$ valid for $C>D>0$ and $\Lambda>0$.
\end{proof}

\subsection{The Deuring--Heilbronn phenomenon}

We conclude the preliminaries by recording the following effective zero-repulsion effect due to  Benli, Goel, Twiss, and Zaman, which shows that an exceptional zero in \Cref{zero-free-LH} allows us to control the remaining zeros.

\begin{prop}[GLH + \cite{Benli-Goel-Twiss-Zaman}*{Theorem 1.3}]\label{repulsion}
Assume \glh, and let $\chi$ and $\chi_{1}$ be (not necessarily distinct) Dirichlet characters modulo $q$, with $\chi_{1}$ real and non-principal.
Suppose that $\beta_{1}=1-\lambda_{1}/\L$ is a real zero of $L(s,\chi_{1})$ with $\lambda_{1}<0.1$, and that $\rho=1-\lambda/\L+i\gamma$, $\rho\neq\rho_{1}$, is a zero of $L(s,\chi)$ with $\lambda<\L/2$ and $|\gamma|\leq\L$.
For all $\delta>0$ there exists an effectively computable constant $q_{4}(\delta)$ such that, if $q\geq q_{4}(\delta)$, then
\[\lambda\geq(1-\delta)\log\Big(\frac{\delta}{5\lambda_{1}}\Big).\]
\end{prop}

\begin{proof}
Under their notation, this follows immediately by taking $\theta$ arbitrarily small, $T=\L$, $\varepsilon=1/2$, $B=100$ (see \cite{Benli-Goel-Twiss-Zaman}*{Theorem 2.7}), and $q$ sufficiently large.
\end{proof}

\section{A prime-detecting sum}

In this section we introduce a device that allows us to detect primes in a given residue class.
This reduces \Cref{main-thm} to the estimation of a sum over zeros of Dirichlet $L$-functions.

Given $L,K>0$ with $L>2K$, write $B=L-2K$, and consider the tent function
\[f(x)=\left\{\begin{array}{ll}
0 & \text{ if }x\leq L-2K,\\
x-(L-2K) & \text{ if }L-2K\leq x\leq L-K,\\
L-x & \text{ if }L-K\leq x\leq L,\\
0 & \text{ if }x\geq L.
\end{array}\right.\]
Denote
\[\Sigma:=\sum_{p\equiv a\mmod{q}}\frac{\log p}{p}f(\L^{-1}\log p).\]
Note that the choice of $f$ restricts the summation to primes $p\in [q^{L-2K},q^{L}]$, thus establishing $\Sigma>0$ immediately yields the existence of a prime $p\equiv a\pmod{q}$ with $q^{L-2K}\leq p\leq q^{L}$.

For $B\geq 2$, we have the following estimate due to Heath-Brown \cite{Heath-Brown-Linnik}*{Lemma 13.1}:
\begin{equation}\label{prime-detector-1}
\Big|\frac{\phi(q)}{\L}\Sigma-K^{2}\Big|\leq\sum_{\chi\neq\chi_{0}}\sum_{\rho_{\chi}}|F((1-\rho_{\chi})\L)|+O(\L^{-1}),
\end{equation}
where the inner sum on the right-hand side is over the non-trivial zeros $\rho_{\chi}$ of $L(s,\chi)$, and
\[F(s)=e^{-Bs}\Big(\frac{1-e^{-Ks}}{s}\Big)^{2}.\]
We now use \cref{log-free-density} to bound the contribution of most of the zeros appearing in \eqref{prime-detector-1}.
Assuming \glh, \eqref{eq-log-free-density} implies
\begin{equation}\label{density-strip}
\sum_{\chi\mmod{q}}N(\alpha,T,\chi)\ll_{\varepsilon}q^{(2+\varepsilon)(1-\alpha)}T^{1+\varepsilon}
\end{equation}
for all $\varepsilon>0$, $0\leq\alpha\leq 1$, and $T\geq 1$.
While this bound is weaker than both \eqref{eq-log-free-density} for $\alpha\geq 1/2$ and the trivial bound for $\alpha\leq 1/2$, we prefer it for its uniformity across the entire critical strip.
To bound the sum over zeros, we consider the regions $R_{m,n}$ defined by
\[1-\frac{m+1}{\L}\leq\sigma<1-\frac{m}{\L},\qquad\frac{n}{\L}\leq |t|<\frac{2n}{\L},\]
where $m\geq 0$ is an integer and $n\in\{0,1,2,4,8,\ldots\}$, and the condition on $t$ is replaced by $|t|\leq\L^{-1}$ when $n=0$.
By \eqref{density-strip}, the total number of zeros $\rho\in R_{m,n}$ contributing to the sum in \eqref{prime-detector-1} is $\ll_{\varepsilon}e^{(2+\varepsilon)m}(1+n/\L)^{1+\varepsilon}$, and for each such zero we have
\[|F((1-\rho)\L)|\ll e^{-Bm}\min\Big(K^{2},\frac{1}{m^{2}+n^{2}}\Big).\]
In total, the contribution from the zeros in regions with $\max(m,n)\geq\log\L$ is bounded by
\begin{align*}
\sum_{\max(m,n)\geq\log\L}\frac{e^{-(B-2-\varepsilon)m}}{m^2+n^2}\Big(1+\frac{n}{\L}\Big)^{1+\varepsilon}&\ll_{\varepsilon}\sum_{m\geq\log\L}e^{-(B-2-\varepsilon)m}+\sum_{n\geq\log\L}\frac{(1+n/\L)^{1+\varepsilon}}{n^2}\\
&\ll_{\varepsilon}\L^{-(B-2-\varepsilon)}+(\log\L)^{-2},
\end{align*}
provided $B>2+\varepsilon$.
We have thus proved the following estimate.

\begin{prop}\label{primes-from-zeros}
Assume \glh.
Given $\varepsilon>0$, set $K=\varepsilon$ and $L=2+4\varepsilon$, so that $B=2+2\varepsilon$.
Then, with the above notation, we have
\begin{equation}
\Big|\frac{\phi(q)}{\L}\Sigma-\varepsilon^{2}\Big|\leq\sum_{\chi\neq\chi_{0}}\sum_{\rho\in S(\chi)}|F((1-\rho)\L)|+O_{\varepsilon}\big((\log\L)^{-2}\big),
\end{equation}
where $S(\chi)$ is the multiset of zeros of $L(s,\chi)$ in the region
\begin{equation}\label{region-S}
1-\frac{\log\L}{\L}\leq\sigma\leq 1,\qquad |t|\leq\frac{\log\L}{\L}.
\end{equation}
\end{prop}

We conclude this section with a bound for an auxiliary sum over $S(\chi)$ that arises naturally in the proof of \Cref{main-thm}.

\begin{lem}\label{bound-sum-zeros}
Let $\delta>0$ be given, and let $\chi$ be a non-principal character modulo $q$.
Then, for $q\geq q_{5}(\delta)$, we have
\[\sum_{\rho\in S(\chi)}|e^{B(1-\rho)\L}F((1-\rho)\L)|\leq K^{2}+\frac{K}{3}+\delta.\]
\end{lem}

\begin{proof}
Heath-Brown \cite{Heath-Brown-Linnik}*{Lemma 13.3} establishes an upper bound of $K^{2}+K\phi+\delta$ for this sum, where $\phi\leq 1/3$ is a parameter depending on $\chi$.
The result then follows.
\end{proof}

\section{\texorpdfstring{Proof of \Cref{main-thm}}{Proof of Theorem 1.1}}

Fix $0<\varepsilon<0.1$.
We adopt the parameter choices $K=\varepsilon$ and $L=2+4\varepsilon$ from \cref{primes-from-zeros}, which yield $B=2+2\varepsilon$.
It therefore suffices to show that the sum
\begin{equation}\label{sum-to-bound}
\sum_{\chi\neq\chi_{0}}\sum_{\rho\in S(\chi)}|F((1-\rho)\L)|
\end{equation}
is strictly less than $\varepsilon^{2}$ for $q$ large enough.
Taking $\eta=\eta(\varepsilon)$ sufficiently large, say $\eta\geq 7\log(1/\varepsilon)$, \Cref{zero-free-LH} ensures that, for $q\geq q_{2}(\eta)$, the function $\prod_{\chi\mmod{q}}L(s,\chi)$ has at most one zero in the region $\sigma\geq 1-\eta/\L$, $|t|\leq \L$.
We now split the argument into two cases.

\subsection{Non-exceptional case}

If there is no such exceptional zero, let $\chi^{(1)},\ldots,\chi^{(N)}$ be the distinct characters modulo $q$ whose $L$-function has a zero in the region \eqref{region-S}, and for each $\chi^{(k)}$ let $\rho^{(k)}=1-\lambda^{(k)}/\L+i\gamma^{(k)}$ be a zero in $S(\chi^{(k)})$ with largest real part.
Then \Cref{bound-sum-zeros} yields
\begin{equation}\label{bound-S-chik}
\sum_{\rho\in S(\chi^{(k)})}|F((1-\rho)\L)|\leq |e^{B(\rho^{(k)}-1)\L}|\sum_{\rho\in S(\chi^{(k)})}|e^{B(1-\rho)\L}F((1-\rho)\L)|\leq 0.05e^{-B\lambda^{(k)}}.
\end{equation}
Moreover, $B\geq 2$ and $\lambda^{(k)}>\eta$ for all $k$, thus \cref{bound-Nlambda} and \eqref{bound-S-chik} imply the sum \eqref{sum-to-bound} is
\[\leq 0.05\sum_{k=1}^{N}e^{-B\lambda^{(k)}}\leq 0.05e^{-1.9\eta}\sum_{k=1}^{N}e^{-0.1\lambda^{(k)}}\leq e^{-1.9\eta}\leq 0.9\varepsilon^{2}\]
for $q$ large enough.

\subsection{Exceptional case}

Suppose now that the exceptional zero $\rho_{1}=1-\lambda_{1}/\L$ does exist.
By a well-known result of Heath-Brown \cite{Heath-Brown-Siegel-Linnik}*{Corollary 2}, the desired bound $P(a,q)\leq q^{2+\varepsilon}$ is already known to hold when $\lambda_{1}$ is sufficiently small in terms of $\varepsilon$.
The use of this result is what ultimately makes our main theorem ineffective.
To obtain an effectively computable constant in this regime, one can instead appeal to \cite{Heath-Brown-Siegel-Linnik}*{Corollary 1}, which yields the weaker bound $P(a,q)\leq q^{3+\varepsilon}$.
We may therefore restrict our attention to the case when $\lambda_{1}\gg_{\varepsilon} 1$.

We label the characters $\chi^{(k)}$ and zeros $\rho^{(k)}$ as before, except that now we do it by taking $\rho_{1}$ to be a zero of $\chi^{(1)}$.
Additionally, if $S(\chi^{(1)})\neq\{\rho_{1}\}$, we choose $\rho^{(1)}$ to satisfy
\[\Re(\rho^{(1)})=\max\{\Re(\rho):\rho\in S(\chi^{(1)}),~\rho\neq\rho_{1}\},\]
and if $S(\chi^{(1)})=\{\rho_{1}\}$ the term $k=1$ is omitted from the sum below.
Consider now $\lambda^{*}=\min\{\lambda^{(k)}\}$.
Arguing as in the non-exceptional case, we obtain
\begin{equation}\label{sum-to-bound-excep}
\sum_{\chi\neq\chi_{0}}\sum_{\rho\in S(\chi)}|F((1-\rho)\L)|\leq |F((1-\rho_{1})\L)|+0.05e^{-1.9\lambda^{*}}\sum_{k=1}^{N}e^{-0.1\lambda^{(k)}}\leq\varepsilon^{2}e^{-2\lambda_{1}}+e^{-1.9\lambda^{*}}
\end{equation}
for $q$ large enough.
Suppose now that $\lambda_{1}<\varepsilon^{10}$, so that $\lambda_{1}<0.1$ and
\begin{equation}\label{eq-lambda1-eps}
\frac{3}{4}\log\Big(\frac{1}{\lambda_{1}}\Big)\geq 2\log\Big(\frac{1}{\varepsilon^{2}}\Big).
\end{equation}
Using \Cref{repulsion} with $\delta$ sufficiently small yields
\begin{equation}\label{lambda-star}
\lambda^{*}\geq\frac{3}{4}\log\Big(\frac{1}{\lambda_{1}}\Big)
\end{equation}
for $q$ large enough, which together with \eqref{eq-lambda1-eps} gives
\begin{equation}\label{bound-constant}
e^{-0.5\lambda^{*}}\leq\varepsilon^{2}.
\end{equation}
Moreover, as $\frac{3}{4}\times 1.4=1.05>1$, \eqref{lambda-star} implies
\begin{equation}\label{bound-repulsion}
e^{-1.4\lambda^{*}}\leq\lambda_{1}.
\end{equation}
Combining \eqref{sum-to-bound-excep}, \eqref{bound-constant}, and \eqref{bound-repulsion}, we conclude
\[\sum_{\chi\neq\chi_{0}}\sum_{\rho\in S(\chi)}|F((1-\rho)\L)|\leq\varepsilon^{2}(e^{-2\lambda_{1}}+\lambda_{1}),\]
and since
\begin{equation}\label{bound-exp}
e^{-x}<1-x+\frac{x^{2}}{2}\qquad\text{ for }x>0,
\end{equation}
we get
\[\sum_{\chi\neq\chi_{0}}\sum_{\rho\in S(\chi)}|F((1-\rho)\L)|<\varepsilon^{2}\big(1-\lambda_{1}(1-2\lambda_{1})\big)
.\]
As $1\ll_{\varepsilon}\lambda_{1}<0.1$, the right-hand side above is strictly less than $\varepsilon^{2}$ for $q$ large enough, which concludes this case.
Lastly, if $\lambda_{1}\geq\varepsilon^{10}$, then \eqref{bound-exp} yields
\begin{equation}\label{exceptional-far}
e^{-2\lambda_{1}}<e^{-2\varepsilon^{10}}<1-2\varepsilon^{10}+2\varepsilon^{20}<1-\varepsilon^{10},
\end{equation}
and since $\lambda^{*}\geq\eta$, from \eqref{sum-to-bound-excep} and \eqref{exceptional-far} we see that
\[\sum_{\chi\neq\chi_{0}}\sum_{\rho\in S(\chi)}|F((1-\rho)\L)|<\varepsilon^{2}-\varepsilon^{12}+e^{-1.9\eta}<\varepsilon^{2}\]
for $q$ large enough, and the proof is complete.

\section*{Acknowledgments}

I would like to thank Asif Zaman for suggesting this problem and for many enlightening discussions, and Andrew Granville for his comments and valuable advice.
Part of this project was completed during my visit to Universit\'e de Montr\'eal, and I am grateful for their hospitality.

\end{document}